\documentclass[preprint,12pt]{elsarticle}

\usepackage{natbib}
\usepackage{graphicx}
\usepackage{multirow}
\usepackage{amsmath,amssymb,amsfonts}
\usepackage{amsthm}
\usepackage{booktabs}
\usepackage{algorithm}
\usepackage{algpseudocode}
\usepackage{subcaption}
\usepackage{float}
\usepackage{placeins}

\newtheorem{theorem}{Theorem}
\newtheorem{proposition}{Proposition}
\newtheorem{corollary}{Corollary}
\newtheorem*{remark}{Remark}

\theoremstyle{definition}

\begin{document}

\begin{frontmatter}

\title{Capacity Control in a Single-Server
Diagnostic--Treatment Queue with Heterogeneous Treatment Modes}
\author[a]{Shobha Rani\corref{cor1}}
\ead{shobharani@iitk.ac.in}

\author[a]{Rohit Dehru}
\ead{rohit.dehru@iitk.ac.in}

\author[a]{Raghu Nandan Sengupta}
\ead{rns@iitk.ac.in}

\address[a]{Department of Management Sciences,
Indian Institute of Technology Kanpur,
Uttar Pradesh 208016, India}
\begin{abstract}
This study develops an analytical and decision framework for a pooled
single-server diagnostic--treatment service receiving both first-time
and referred patients. A first-time patient receives diagnosis followed
immediately by one of $n$ treatment modes, whereas a referred patient
enters directly into the prescribed treatment mode. For the two-treatment case, closed-form
probability generating functions characterize the stationary queue-length
distribution; for arbitrary $n$, a matrix-geometric representation provides
a computational solution. The queueing
analysis is then directed to capacity control decisions. The effective load is
decomposed into mode-specific workload contributions, yielding a
minimum service-rate increment for a prescribed utilization target to show
that utilization-equivalent interventions need not be delay-equivalent. This motivates a
distribution-aware control based on the stationary probability of severe
congestion. Finally, the total cost function shown to be strictly
convex, providing a global minimum
for restricted one-mode interventions. Analytical results are
cross-validated against the Pollaczek--Khintchine relations and discrete-event
simulation. 
\end{abstract}


\end{frontmatter}

\section{Introduction}\label{sec1}

Healthcare services often operate under a common capacity while receiving
patients with different service requirements. This feature is especially
important in referral-based care. A first-time patient may require an initial
diagnostic assessment followed immediately by treatment, whereas a referred
patient arrives with a diagnosis and requires treatment only. If treatment
requirements are themselves heterogeneous, the amount of work brought by an
arrival depends on the patient type, the post-diagnostic routing decision, and
the service time of the prescribed treatment. The resulting congestion is
therefore determined by the composition of demand as well as by its volume.
These features are relevant to specialist clinics and other pooled services in
which one provider alternates between diagnostic and treatment activities.

This paper studies such a system as a single pooled FCFS service channel.
First-time patients undergo diagnosis and then, without releasing the server
or rejoining the queue, continue directly to one of $n$ treatment modes.
Referred patients bypass diagnosis and begin with their prescribed treatment
mode. The mode-specific rates describe the effective speed of the same pooled
provider while performing different activities; they do not represent
parallel doctors or departments. The external arrivals are exogenous, and the
focus is on stationary congestion and on planning decisions that improve one
or more mode-specific service rates. This setting differs from referral models
that primarily optimize patient routing, gatekeeping, or strategic joining
\citep{liu2015_mutual_referral_policy,li2020_referral_coordination_alliance,
yu2022_referral_tiered,liu2024_strategic_queueing_patients}.

The patient pathway produces a phase-type service requirement and hence an
$M/PH/1$ queue. The first two service-time moments are sufficient for the
Pollaczek--Khintchine mean-value relations, but the capacity decisions studied
here require more than mean delay alone. For the two-treatment case, we obtain
closed-form probability generating functions for the stationary queue-length
distribution. For arbitrary $n$, a matrix-geometric representation gives the
stationary probabilities and the queue tail. The two analytical views are
therefore complementary: the moment representation provides transparent mean
performance expressions, while the distributional representation quantifies
the probability of severe congestion.

The decision analysis is built on this queueing anlysis. The
utilization is first decomposed into mode-specific workload contributions,
which identifies how much of the provider's load is generated by diagnosis and
by each treatment mode, yields a minimum
service-rate increment for attaining a prescribed utilization target. A
utilization target, however, fixes only the first service-time moment. Two
mode-specific interventions can attain the same utilization and still produce
different second moments, mean delays, and queue tails. We therefore combine
the closed-form load adjustment with a distribution-aware congestion check.
The expected total cost is constructed, becomes strictly convex after transforming service
rates to reciprocal service times. 

The study makes four contributions. First, it formulates the pooled
diagnostic--treatment referral pathway as a structured phase-type service
model, derives an exact two-treatment PGF representation, and provides an
implementable matrix-geometric solution for arbitrary $n$. Second, it develops
an additive workload decomposition and a closed-form phase-targeted capacity
rule, including feasibility and break-even cost conditions. Third, it shows that utilization-equivalent interventions need not be
performance-equivalent and uses the stationary queue tail to construct a
minimal distribution-aware intervention. Fourth, it establishes strict
convexity of the total cost in reciprocal service-time variables
and quantifies the cost of restricting capacity adjustment to a single mode.
The analytical results are cross-validated through independent moment
relations and discrete-event simulation.

The remainder of the paper is organized as follows. Section~\ref{sec2}
reviews the healthcare queueing, referral, matrix-analytic, service-rate
control, and simulation literature relevant to the model. Section~\ref{sec:model}
develops the phase-type representation, the stationary analysis, the
arbitrary-$n$ matrix-geometric solution, and its validation.
Section~\ref{sec:structural} studies the structural drivers of congestion.
Section~\ref{sec:decision} develops the utilization-targeted,
distribution-aware, and economic capacity decisions. Section~\ref{sec:conclusion}
concludes with the main findings, limitations, and possible extensions.

\section{Literature Review}
\label{sec2}

Queueing analysis has a long history in healthcare operations because delay,
resource utilization, and patient flow are jointly shaped by stochastic
arrivals and uncertain service requirements. Classical applications include
hospital waiting lists, outpatient services, admission control, and operating
room planning \citep{worthington1987_waiting_list_models,green2006_queueing,
helm2011_hospital_admission_control,gupta2007_operating_room_scheduling,
palvannan2012_queueing_healthcare}. More recent work has emphasized that
healthcare queues often contain operational features that are not represented
well by simple single-class birth--death models. Examples include discharge
delays and time-varying patient-flow mechanisms in inpatient wards
\citep{dong2020_patient_flow_wards}, as well as practical queue-management and
waiting-list settings in which congestion must be interpreted together with
service organization \citep{pazinfilho2024_surgical_waiting}. These studies
motivate analytical models that retain a clear operational interpretation
rather than treating all patient service requirements as homogeneous.

Referral systems introduce an additional source of structure because patients
may enter care at different stages or arrive with different information about
the required service. The referral literature has examined coordination among
facilities, threshold referral policies, gatekeeping, and capacity decisions
across tiers of care \citep{liu2015_mutual_referral_policy,
li2020_referral_coordination_alliance,yu2022_referral_tiered}. Strategic
queueing models have also considered the joining behaviour of first-time and
referred patients when their incentives differ
\citep{liu2024_strategic_queueing_patients}. The present setting addresses a
different operational mechanism. Both patient groups join the same FCFS
provider, but they bring different work content: a first-time patient consumes
a diagnostic phase followed by treatment, while a referred patient consumes
only the prescribed treatment. The treatment modes may have different service
rates, so referral status and treatment heterogeneity enter the service-time
distribution together.

Phase-type distributions are well suited to sequential service requirements
because they preserve Markovian tractability while allowing a service time to
contain several transient phases. Matrix-geometric and matrix-analytic methods
provide the corresponding stationary machinery for QBD-type processes
\citep{neuts1981matrixgeometric,latoucheramaswami1999introduction,
he2014fundamentals}. Standard queueing references also establish the
$M/G/1$ and stochastic-recursion results used for the mean-value and stability
analysis \citep{gross2011_fundamentals_queueing_theory,
baccelli2013_elements_queueing}. In addition to mean queue length and delay,
distributional queueing results are useful when an operational decision
depends on the probability of a large queue rather than on its mean alone
\citep{keilson1988distributional,wolff1982pasta}. In the present model, the
matrix-geometric representation serves this distributional role: it extends
the exact two-treatment analysis to arbitrary $n$ and supplies the stationary
queue tail used in the capacity intervention.

The decision component is related to the literature on service-rate control
and economic design of queues. Phase-based service-rate control in $M/G/1$ systems has been studied using
phase methods \citep{jo1983service_rate}, while broader structural results for
controlled queues and service-rate decisions are available in the classical
control literature \citep{stidham1985optimal,weberstidham1987optimal}. Related models
show how the preferred service rate can depend on the balance between delay
costs and capacity costs \citep{bradley2005_optimal,bouchentouf2019_cost}.
Our decision problem is not a continuously state-dependent control problem.
Instead, it is a planning problem in which the effective speed of diagnosis or
of a treatment mode can be improved while the same pooled provider continues
to process one patient at a time. This distinction makes it possible to use
the workload decomposition to obtain a closed-form targeted intervention and,
for joint redesign, to expose a convex structure in reciprocal service-time
variables using standard perspective and KKT arguments
\citep{boydvandenberghe2004convex}.

Simulation is widely used in healthcare operations both for evaluating patient
flow and for validating analytical approximations. General principles for the design, verification, warm-up, replication, and
output analysis of discrete-event simulation experiments are well established
\citep{law2015simulation}, and healthcare applications illustrate their use for
queue-management and
length-of-stay analysis \citep{tan2013_dynamic_queue_management_ed}. In this
study, simulation is used only as an independent numerical validation of the
matrix-analytic results; the capacity decisions themselves are obtained from
the analytical model.

Taken together, these streams provide strong methods for healthcare queueing,
referral coordination, phase-type analysis, service-rate control, and
simulation. The contribution developed here lies in connecting these elements
for one pooled diagnostic--treatment provider. The phase structure is used to
separate diagnosis and treatment workload, the workload decomposition yields a
closed-form targeted rate adjustment, the stationary distribution checks the
remaining congestion risk, and the convex cost model measures the economic
consequence of targeted rather than joint redesign. This connection between
queue structure and implementable capacity decisions is the organizing theme
of the analysis that follows.

\section{Model formulation}
\label{sec:model}
We consider a pooled hospital service in which newly arriving
and referred patients wait in a common FCFS
queue. The hospital is represented by a single aggregate service
channel that processes one patient service phase at a time. Depending on the requirements of the patient currently in service, the service
channel operates either in a diagnostic mode or in one of $n$
heterogeneous treatment modes.
Patients are classified into two groups. The first group consists
of newly arriving, undiagnosed patients. Such patients first undergo diagnosis and then, without releasing the
aggregate service channel or rejoining the queue, continue directly to
one of the $n$ treatment modes. Thus, diagnosis and the subsequent
treatment constitute one uninterrupted patient service requirement,
and FCFS applies at the patient level.
\begin{figure}[htbp]
    \centering
    \includegraphics[width=1.0\linewidth]{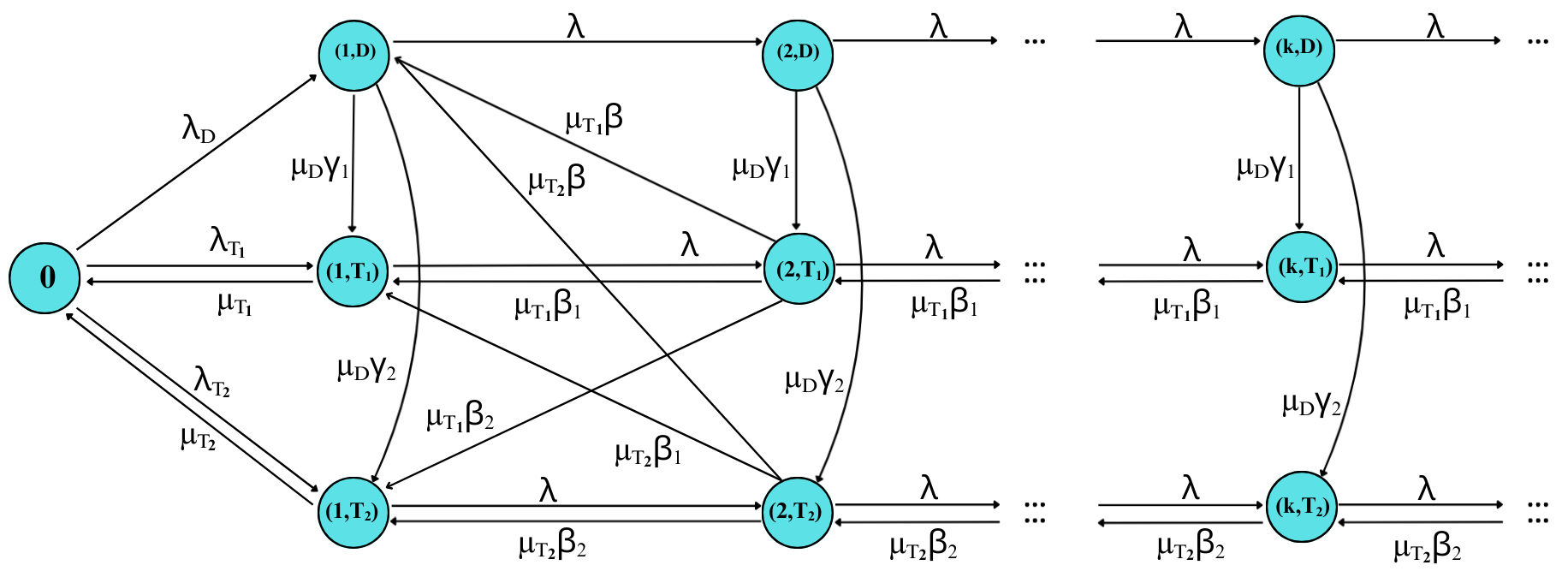}
    \caption{Rate-transition diagram for the pooled diagnostic--treatment
    queue with two heterogeneous treatment modes.}
    \label{fig:rate_transition}
\end{figure}
The model represents one aggregate provider rather than parallel diagnostic and treatment servers. The rates $\mu_D$ and $\mu_{T_i}$ are therefore mode-specific effective service rates of the same pooled service channel. An increase in one of these rates represents a reduction in the mean time required in that service mode, for example through workflow redesign, decision support, equipment, or standardized procedures; it does not create simultaneous service capacity. We assume an infinite waiting room, non-preemptive FCFS service at the patient level, no abandonment or priority classes, stationary independent Poisson arrival streams, mutually independent exponential service phases, and treatment type known on arrival for referred patients.

The arrival processes of new and referred patients follow Poisson processes. Let $\lambda_D$ denote the
arrival rate of new patients, and let $\lambda_{T_i}$ denote the
arrival rate of referred patients requiring treatment mode
$T_i$, $i=1,\ldots,n$. By the superposition property of independent Poisson processes, the
total external patient arrival rate is
$\lambda = \lambda_D + \sum_{i=1}^{n} \lambda_{T_i}$. Define
$\beta = \lambda_D/\lambda$ and $\beta_i = \lambda_{T_i}/\lambda$
for $i = 1, \ldots, n$, so that
$\beta + \sum_{i=1}^{n} \beta_i = 1$. Here, $\beta$ denotes the
fraction of new patients, whereas $\beta_i$ denotes the fraction of
referred patients whose required treatment is of type $i$. After
completing diagnosis, a new patient is routed to treatment mode $T_i$
with probability $\gamma_i$, where $\gamma_i \geq 0$ and
$\sum_{i=1}^{n} \gamma_i = 1$. For later use, define the effective treatment weights
$\omega_i=\beta_i+\beta\gamma_i$, $i=1,\ldots,n$, and set
$\omega_D\equiv\beta$. Thus $\sum_{i=1}^n\omega_i=1$, while
$\omega_D+\sum_i\omega_i=1+\beta$ because each first-time patient consumes both a diagnostic and a treatment phase.
Let $D\sim\operatorname{Exp}(\mu_D)$ denote the diagnostic duration and let
$Y_i\sim\operatorname{Exp}(\mu_{T_i})$ denote the duration of treatment mode $T_i$. All service phases and arrival processes are
assumed to be mutually independent.
The service requirement $S$ of
an arriving patient can then be represented as
\begin{equation}
S=
\begin{cases}
Y_i,
&
\text{with probability }\beta_i,
\\[1mm]
D+Y_i,
&
\text{with probability }\beta\gamma_i,
\end{cases}
\qquad i=1,\ldots,n.
\label{eq:service_time_mixture}
\end{equation}
Thus, the service-time distribution is phase-type, and
the system can be classified as an $M/PH/1$ queue. A phase-type representation is obtained by defining the initial phase
probability vector
\begin{equation}
\boldsymbol{\alpha}
=
\begin{pmatrix}
\beta & \beta_1 & \beta_2 & \cdots & \beta_n
\end{pmatrix}
\label{eq:ph_initial_vector}
\end{equation}
and the transient generator
\begin{equation}
\mathbf{T}
=
\begin{pmatrix}
-\mu_D
&
\mu_D\gamma_1
&
\mu_D\gamma_2
&
\cdots
&
\mu_D\gamma_n
\\
0
&
-\mu_{T_1}
&
0
&
\cdots
&
0
\\
0
&
0
&
-\mu_{T_2}
&
\cdots
&
0
\\
\vdots
&
\vdots
&
\vdots
&
\ddots
&
\vdots
\\
0
&
0
&
0
&
\cdots
&
-\mu_{T_n}
\end{pmatrix}.
\label{eq:ph_subgenerator}
\end{equation}
The corresponding exit vector is $\mathbf t=-\mathbf T\mathbf 1=(0,\mu_{T_1},\ldots,\mu_{T_n})^{\top}$. The zero diagnostic component reflects the patient pathway: a first-time patient cannot depart immediately after diagnosis and must continue to a treatment phase.
A referred type-$i$ patient begins service directly in phase $T_i$,
whereas a new patient begins in the diagnostic phase and subsequently
moves to phase $T_i$ with probability $\gamma_i$. The mean service requirement of an arbitrary patient is
\begin{equation}
\mathbb{E}[S]
=
\frac{\beta}{\mu_D}
+
\sum_{i=1}^{n}
\frac{\omega_i}{\mu_{T_i}},
\label{eq:mean_service_general}
\end{equation}
and the effective load is therefore
\begin{equation}
\rho
=
\lambda\mathbb{E}[S]
=
\lambda
\left[
\frac{\beta}{\mu_D}
+
\sum_{i=1}^{n}
\frac{\omega_i}{\mu_{T_i}}
\right].
\label{eq:rho_general}
\end{equation}
The stationary analysis is restricted to the positive-recurrent
region $\rho<1$. For explicit analytical derivation, we first consider the
two-treatment case, $n=2$. In this case,
\begin{equation}
\lambda
=
\lambda_D+\lambda_{T_1}+\lambda_{T_2},
\end{equation}
and
\begin{equation}
\beta
=
\frac{\lambda_D}{\lambda},
\qquad
\beta_1
=
\frac{\lambda_{T_1}}{\lambda},
\qquad
\beta_2
=
\frac{\lambda_{T_2}}{\lambda},
\end{equation}
with
\begin{equation}
\beta+\beta_1+\beta_2=1,
\qquad
\gamma_1+\gamma_2=1.
\end{equation}
Let $N(t)$ denote the total number of patients in the system at time
$t$, including the patient in service, and let $J(t)$ denote the
service phase of the patient currently receiving service. For the
two-treatment case,
\[
J(t) =
\begin{cases}
D, & \text{diagnostic phase},\\
T_1, & \text{treatment mode }T_1,\\
T_2, & \text{treatment mode }T_2.
\end{cases}
\]
The process $X(t)=\bigl(N(t),J(t)\bigr)$ is a continuous-time Markov chain (CTMC) with state space
\[
\mathcal{S}
=
\{0\}
\cup
\bigl\{(k,j):k\geq1,\ j\in\{D,T_1,T_2\}\bigr\}.
\]
For $k \geq 1$, define
\[
p_{k,j}(t)
=
\mathbb{P}\bigl\{N(t)=k,\;J(t)=j\bigr\},
\]
and let
$p_{k,j} = \lim_{t \to \infty} p_{k,j}(t)$ denote the corresponding
stationary probabilities. The balance equations are derived from the transition structure shown
in Figure~\ref{fig:rate_transition}, and the
probability generating functions are subsequently used to obtain the
stationary queue-length distribution and key performance
measures.
\subsection{Queue size distribution}
\label{subsec:queue_size_distribution}
To evaluate the stationary probability, the
governing balance equations of the system are defined as follows:
\begin{equation}
\lambda\,p_{0}
= \mu_{T_1}\,p_{(1,T_{1})}
+ \mu_{T_2}\,p_{(1,T_{2})}.
\label{eq:boundary_balance}
\end{equation}
\subsection*{For diagnosis:}
\begin{equation}
(\lambda + \mu_{D})\,p_{(1,D)}
= \lambda\,\beta\,p_{0}
+ \mu_{T_1}\,\beta\,p_{(2,T_{1})}
+ \mu_{T_2}\,\beta\,p_{(2,T_{2})}
\label{diagnose:1}
\end{equation}
\begin{equation}
(\lambda + \mu_{D})\,p_{(k,D)}
= \lambda\,p_{(k-1,D)}
+ \mu_{T_1}\,\beta\,p_{(k+1,T_{1})}
+ \mu_{T_2}\,\beta\,p_{(k+1,T_{2})},
\quad k \ge 2
\label{diagnose:n}
\end{equation}
\subsection*{For treatment 1:}
\begin{equation}
(\lambda + \mu_{T_1})\,p_{(1,T_{1})}
= \lambda\,\beta_{1}\,p_{0}
+ \mu_{D}\,\gamma_{1}\,p_{(1,D)}
+ \mu_{T_1}\,\beta_{1}\,p_{(2,T_{1})}
+ \mu_{T_2}\,\beta_{1}\,p_{(2,T_{2})}
\label{treatment1:1}
\end{equation}
\begin{equation}
(\lambda + \mu_{T_1})\,p_{(k,T_{1})}
= \lambda\,p_{(k-1,T_{1})}
+ \mu_{D}\,\gamma_{1}\,p_{(k,D)}
+ \mu_{T_1}\,\beta_{1}\,p_{(k+1,T_{1})}
+ \mu_{T_2}\,\beta_{1}\,p_{(k+1,T_{2})},
\quad k \ge 2
\label{treatment1:n}
\end{equation}
\subsection*{For treatment 2:}
\begin{equation}
(\lambda + \mu_{T_2})\,p_{(1,T_{2})}
= \lambda\,\beta_{2}\,p_{0}
+ \mu_{D}\,\gamma_{2}\,p_{(1,D)}
+ \mu_{T_2}\,\beta_{2}\,p_{(2,T_{2})}
+ \mu_{T_1}\,\beta_{2}\,p_{(2,T_{1})}
\label{treatment2:1}
\end{equation}
\begin{equation}
(\lambda + \mu_{T_2})\,p_{(k,T_{2})}
= \lambda\,p_{(k-1,T_{2})}
+ \mu_{D}\,\gamma_{2}\,p_{(k,D)}
+ \mu_{T_2}\,\beta_{2}\,p_{(k+1,T_{2})}
+ \mu_{T_1}\,\beta_{2}\,p_{(k+1,T_{1})},
\quad k \ge 2
\label{treatment2:n}
\end{equation}
\subsubsection{PGF characterization}
\label{subsec:pgf_characterization}
To characterize the stationary queue-length distribution, define the probability generating functions
\begin{equation}
P_D(z)
=
\sum_{k=1}^{\infty}p_{(k,D)}z^k,
\qquad
P_{T_i}(z)
=
\sum_{k=1}^{\infty}p_{(k,T_i)}z^k,
\quad i=1,2,
\label{eq:phase_pgfs}
\end{equation}
for $|z|<1$. Their values at $z=1$, give
the stationary probabilities that the server is operating in the
corresponding service phase. Multiplying the balance equations by the appropriate
powers of $z$ and summing over the queue length yields
\begin{align}
A_0(z)P_D(z)
&=
C_0(z)p_0
+
L_1P_{T_1}(z)
+
L_2P_{T_2}(z),
\label{eq:pgf_D_compact}
\\
A_1(z)P_{T_1}(z)
&=
C_1(z)p_0
+
R_1(z)P_D(z)
+
K_1P_{T_2}(z),
\label{eq:pgf_T1_compact}
\\
A_2(z)P_{T_2}(z)
&=
C_2(z)p_0
+
R_2(z)P_D(z)
+
K_2P_{T_1}(z),
\label{eq:pgf_T2_compact}
\end{align}
where
\begin{align}
A_0(z)
&=
z\lambda(1-z)+z\mu_D,
&
C_0(z)
&=
z\lambda\beta(z-1),
\nonumber\\
A_1(z)
&=
z\lambda(1-z)+\mu_{T_1}(z-\beta_1),
&
C_1(z)
&=
z\lambda\beta_1(z-1),
\nonumber\\
A_2(z)
&=
z\lambda(1-z)+\mu_{T_2}(z-\beta_2),
&
C_2(z)
&=
z\lambda\beta_2(z-1),
\label{eq:pgf_coefficients_1}
\\
R_1(z)
&=
z\mu_D\gamma_1,
&
R_2(z)
&=
z\mu_D\gamma_2,
\nonumber\\
L_1
&=
\mu_{T_1}\beta,
&
L_2
&=
\mu_{T_2}\beta,
\nonumber\\
K_1
&=
\mu_{T_2}\beta_1,
&
K_2
&=
\mu_{T_1}\beta_2.
\label{eq:pgf_coefficients_2}
\end{align}
Equations~\eqref{eq:pgf_D_compact}--\eqref{eq:pgf_T2_compact}
can be represented as the linear system
\begin{equation}
\mathbf{M}(z)\mathbf{P}(z)
=
p_0\mathbf{C}(z),
\label{eq:pgf_matrix_system}
\end{equation}
where
\begin{equation}
\mathbf{P}(z)
=
\begin{pmatrix}
P_D(z)\\
P_{T_1}(z)\\
P_{T_2}(z)
\end{pmatrix},
\qquad
\mathbf{C}(z)
=
\begin{pmatrix}
C_0(z)\\
C_1(z)\\
C_2(z)
\end{pmatrix},
\label{eq:pgf_vectors}
\end{equation}
and
\begin{equation}
\mathbf{M}(z)
=
\begin{pmatrix}
A_0(z) & -L_1 & -L_2\\
-R_1(z) & A_1(z) & -K_1\\
-R_2(z) & -K_2 & A_2(z)
\end{pmatrix}.
\label{eq:pgf_coefficient_matrix}
\end{equation}
The matrix representation separates the queue length dynamics and provides a basis for both the
closed-form analysis for $n=2$ and the numerical
extension to an arbitrary $n$.
\begin{theorem}
\label{thm:stationary_pgfs}
The PGFs corresponding to the number of patients during diagnosis, treatment 1 and treatment 2 are explicitly obtained as
\begin{align}
P_D(z)
&=
\frac{
\beta\lambda z
\bigl(\lambda+\mu_{T_1}-\lambda z\bigr)
\bigl(\lambda+\mu_{T_2}-\lambda z\bigr)
}{
H(z)
}\,p_0,
\label{eq:PD_closed}
\\[1ex]
P_{T_1}(z)
&=
\frac{
\lambda z
\bigl(\lambda+\mu_{T_2}-\lambda z\bigr)
\left[
\beta_1\bigl(\lambda+\mu_D-\lambda z\bigr)
+
\beta\gamma_1\mu_D
\right]
}{
H(z)
}\,p_0,
\label{eq:PT1_closed}
\\[1ex]
P_{T_2}(z)
&=
\frac{
\lambda z
\bigl(\lambda+\mu_{T_1}-\lambda z\bigr)
\left[
\beta_2\bigl(\lambda+\mu_D-\lambda z\bigr)
+
\beta\gamma_2\mu_D
\right]
}{
H(z)
}\,p_0,
\label{eq:PT2_closed}
\end{align}
where, on defining $x(z)=\lambda(1-z)$, the common denominator is
\begin{align}
H(z)
={}&
x(z)^3
+
\left(
\mu_D+\mu_{T_1}+\mu_{T_2}-\lambda
\right)x(z)^2
\nonumber\\
&+
\Big[
\mu_D\mu_{T_1}
+
\mu_D\mu_{T_2}
+
\mu_{T_1}\mu_{T_2}
\nonumber\\
&\hspace{10mm}
-\lambda
\bigl\{
\mu_D
+
\mu_{T_1}(1-\beta_1)
+
\mu_{T_2}(1-\beta_2)
\bigr\}
\Big]x(z)
\nonumber\\
&+
\mu_D\mu_{T_1}\mu_{T_2}(1-\rho).
\label{eq:H_denominator}
\end{align}
\end{theorem}
\begin{proof}
The determinant of the coefficient matrix in
\eqref{eq:pgf_coefficient_matrix} is
\begin{align}
\Delta(z)
={}&
A_0A_1A_2
-
A_0K_1K_2
-
A_2L_1R_1
-
A_1L_2R_2
\nonumber\\
&-
K_1L_1R_2
-
K_2L_2R_1.
\label{eq:pgf_determinant}
\end{align}
Using $\beta+\beta_1+\beta_2=1$ and
$\gamma_1+\gamma_2=1$, together with
$\mu_D\mu_{T_1}\mu_{T_2}m_1=\mu_{T_1}\mu_{T_2}\beta+\mu_D\mu_{T_2}\omega_1+\mu_D\mu_{T_1}\omega_2$, direct factorization gives
\begin{equation}
\Delta(z)=z^2(z-1)H(z).
\label{eq:determinant_factorization}
\end{equation}
Applying Cramer's rule to \eqref{eq:pgf_matrix_system}, factorizing the
three numerator determinants, and cancelling the common factor
$z^2(z-1)$ yields
\eqref{eq:PD_closed}--\eqref{eq:PT2_closed}.
\end{proof}
\subsection{Performance measures}
\label{sec:performance-measures}
We now derive the principal congestion, delay, throughput, and stability measures. For $n=2$, the treatment weights introduced above satisfy $\omega_1+\omega_2=1$.
The first two moments of the service requirement $S$ of an arbitrary
external patient are
\begin{equation}
m_1
\equiv
\mathbb{E}[S]
=
\frac{\beta}{\mu_D}
+
\frac{\omega_1}{\mu_{T_1}}
+
\frac{\omega_2}{\mu_{T_2}},
\label{eq:first_service_moment}
\end{equation}
and
\begin{align}
m_2
\equiv
\mathbb{E}[S^2]
={}&
\frac{2\beta}{\mu_D^2}
+
\frac{2\beta\gamma_1}{\mu_D\mu_{T_1}}
+
\frac{2\beta\gamma_2}{\mu_D\mu_{T_2}}
\nonumber\\
&+
\frac{2\omega_1}{\mu_{T_1}^2}
+
\frac{2\omega_2}{\mu_{T_2}^2}.
\label{eq:second_service_moment}
\end{align}
The effective load is therefore
\begin{equation}
\rho
=
\lambda m_1
=
\lambda
\left(
\frac{\beta}{\mu_D}
+
\frac{\beta_1+\beta\gamma_1}{\mu_{T_1}}
+
\frac{\beta_2+\beta\gamma_2}{\mu_{T_2}}
\right).
\label{eq:effective_load_performance}
\end{equation}
Under $\rho<1$, the empty-system probability satisfies $p_0=1-\rho$.
\subsubsection{Expected system size}
The expected number of patients in the system, including any patient
in service, is obtained from the derivatives of the PGFs:
\[
L
=
P_D'(1)
+
P_{T_1}'(1)
+
P_{T_2}'(1).
\]
This leads to the closed form expression
\begin{equation}
L
=
\rho
+
\frac{\lambda^2m_2}{2(1-\rho)}.
\label{eq:L_compact}
\end{equation}
Furthermore, the mean waiting time, excluding service, is
\begin{equation}
W_q
=
\frac{\lambda m_2}{2(1-\rho)}.
\label{eq:mean_queue_wait}
\end{equation}
Consequently, the mean sojourn time including both queueing delay and service
can be given as
\begin{equation}
W
=
m_1+W_q
=
m_1+
\frac{\lambda m_2}{2(1-\rho)}.
\label{eq:mean_sojourn_time}
\end{equation}
Accordingly, Little's law gives $L=\lambda W$. Equations~\eqref{eq:L_compact}--\eqref{eq:mean_sojourn_time} are the Pollaczek--Khintchine mean-value relations for the equivalent $M/G/1$ service requirement \citep{gross2011_fundamentals_queueing_theory}. The PGF representation is retained because it characterizes the stationary distribution rather than only its first two moments; the arbitrary-$n$ matrix-geometric extension below uses this distributional information to evaluate congestion-tail probabilities.
\begin{remark}
Equation~\eqref{eq:L_compact} also shows that equal utilization need not imply equal mean delay: when two rate configurations attain the same $\rho$, their ordering in $L$ and $W$ is determined by their second service-time moments. Section~\ref{sec:equal-load} turns this observation into an intervention result.
\end{remark}
\subsubsection{Phase occupancies and throughput}
Evaluating the PGFs at $z=1$ yields the treatment completion rate and the rate of completed service phases, defined as
\begin{align}
TH
&=
\mu_{T_1}P_{T_1}(1)
+
\mu_{T_2}P_{T_2}(1) = \lambda.
\label{eq:patient_throughput}
\end{align}
\begin{align}
TH_{\mathrm{phase}}
&=
\mu_D P_D(1)
+
\mu_{T_1}P_{T_1}(1)
+
\mu_{T_2}P_{T_2}(1)
\nonumber\\
&=
\lambda(1+\beta).
\label{eq:phase_completion_intensity}
\end{align}
$TH$ counts completed patients, whereas $TH_{\mathrm{phase}}$ counts all completed diagnostic and treatment phases.
\subsubsection{Stability boundary}
To investigate the stability boundary, the effective utilization $\rho$ is computed as
\begin{equation}
\rho
=
\frac{\lambda_D}{\mu_D}
+
\frac{\lambda_{T_1}+\gamma_1\lambda_D}{\mu_{T_1}}
+
\frac{\lambda_{T_2}+\gamma_2\lambda_D}{\mu_{T_2}}.
\label{eq:rho_primitive_rates}
\end{equation}
The stationary distribution exists if and only if $\rho<1$. Solving $\rho(\lambda)=1$ yields the critical total arrival rate as
\begin{equation}
\lambda^*
=
\left(
\frac{\beta}{\mu_D}
+
\frac{\beta_1+\beta\gamma_1}{\mu_{T_1}}
+
\frac{\beta_2+\beta\gamma_2}{\mu_{T_2}}
\right)^{-1}.
\label{eq:critical_arrival_rate}
\end{equation}
For arrival rates varying independently rather than proportionally,
the relevant stability boundary is the multidimensional condition
given in Equation~\eqref{eq:rho_primitive_rates}, rather than a single
scalar value $\lambda^*$.
\subsection{Numerical illustration}
\label{subsec:analytical_benchmark}
In this subsection, the analytical performance measures
under progressively increasing arrival intensities are computed (in Table~\ref{tab:validation}). Consequently,
the patient-mix probabilities remain constant, whereas the effective
load $\rho$ increases monotonically from S1 to S8 and approaches the
stability boundary. In all scenarios, the service rates are fixed at
$\mu_D=8$, $\mu_{T_1}=5$, and $\mu_{T_2}=7$, while the routing
probabilities are set to $\gamma_1=0.6$ and $\gamma_2=0.4$. The arrival rates satisfy $\lambda_{T_1}=0.3\lambda_D$,
$\lambda_{T_2}=0.4\lambda_D$. Thus, the patient-mix probabilities are
\[
\beta=\frac{10}{17},
\qquad
\beta_1=\frac{3}{17},
\qquad
\beta_2=\frac{4}{17}.
\]
\begin{table}[htbp]
\centering
\caption{Analytical performance under increasing traffic intensity.}
\label{tab:validation}
\small
\setlength{\tabcolsep}{4.5pt}
\begin{tabular}{c c c c c c c c c}
\toprule
\textbf{Set}
& $\lambda_D$
& $\lambda_{T_1}$
& $\lambda_{T_2}$
& $\lambda$
& $\rho$
& $L$
& $W$
& $TH$ \\
\midrule
S1 & 1.000 & 0.300 & 0.400 & 1.700 & 0.419 & 0.683 & 0.402 & 1.700 \\
S2 & 1.300 & 0.390 & 0.520 & 2.210 & 0.545 & 1.114 & 0.504 & 2.210 \\
S3 & 1.600 & 0.480 & 0.640 & 2.720 & 0.671 & 1.862 & 0.685 & 2.720 \\
S4 & 1.900 & 0.570 & 0.760 & 3.230 & 0.797 & 3.516 & 1.088 & 3.230 \\
S5 & 2.200 & 0.660 & 0.880 & 3.740 & 0.922 & 10.479 & 2.802 & 3.740 \\
S6 & 2.250 & 0.675 & 0.900 & 3.825 & 0.943 & 14.641 & 3.828 & 3.825 \\
S7 & 2.290 & 0.687 & 0.916 & 3.893 & 0.960 & 21.123 & 5.426 & 3.893 \\
S8 & 2.320 & 0.696 & 0.928 & 3.944 & 0.973 & 31.217 & 7.915 & 3.944 \\
\bottomrule
\end{tabular}
\end{table}
The values reported in
Table~\ref{tab:validation} satisfy Little's law, $L=\lambda W$, and
the steady-state flow-conservation identity, $TH=\lambda$. Both $L$
and $W$ increase sharply as $\rho$ approaches one, whereas $TH$
increases only proportionally with the external arrival rate. This
behavior reflects the reduction in available capacity slack near the
stability boundary.
\subsection{Matrix-geometric solution for arbitrary $n$}
\label{subsec:general_n}

The two-treatment structure extends directly to arbitrary $n$ through
the phase-type representation of the service requirement. Let
\begin{equation}
\boldsymbol{\alpha}
=
(\beta,\beta_1,\ldots,\beta_n),
\qquad
\mathbf{T}
=
\begin{pmatrix}
-\mu_D & \mu_D\gamma_1 & \cdots & \mu_D\gamma_n\\
0 & -\mu_{T_1} & \cdots & 0\\
\vdots & \vdots & \ddots & \vdots\\
0 & 0 & \cdots & -\mu_{T_n}
\end{pmatrix},
\qquad
\mathbf{t}
=
-\mathbf{T}\mathbf{1}.
\label{eq:general_PH_triple}
\end{equation}

The first two moments of the patient service requirement are
\begin{equation}
m_1
=
\frac{\beta}{\mu_D}
+
\sum_{i=1}^{n}
\frac{\omega_i}{\mu_{T_i}},
\qquad
m_2
=
\frac{2\beta}{\mu_D^2}
+
2\beta
\sum_{i=1}^{n}
\frac{\gamma_i}{\mu_D\mu_{T_i}}
+
2\sum_{i=1}^{n}
\frac{\omega_i}{\mu_{T_i}^2},
\label{eq:general_moments}
\end{equation}
where
\[
\omega_i=\beta_i+\beta\gamma_i.
\]
Hence the mean performance measures follow directly from the
Pollaczek--Khintchine relations whenever
$\rho=\lambda m_1<1$.

For the stationary queue-length distribution, the corresponding QBD
blocks are
\begin{equation}
\mathbf{B}_{+}
=
\lambda\mathbf{I}_{n+1},
\qquad
\mathbf{B}_{0}
=
\mathbf{T}-\lambda\mathbf{I}_{n+1},
\qquad
\mathbf{B}_{-}
=
\mathbf{t}\boldsymbol{\alpha}.
\label{eq:general_QBD_blocks}
\end{equation}

For this $M/PH/1$ structure, the rate matrix admits the explicit form
\begin{equation}
\mathbf{R}
=
\lambda
\left(
\lambda\mathbf{I}_{n+1}
-
\lambda\mathbf{1}\boldsymbol{\alpha}
-
\mathbf{T}
\right)^{-1},
\label{eq:explicit-R}
\end{equation}
which is the minimal nonnegative solution of
\[
\mathbf{B}_{+}
+
\mathbf{R}\mathbf{B}_{0}
+
\mathbf{R}^{2}\mathbf{B}_{-}
=
\mathbf{0}
\]
for the stable $M/PH/1$ queue
\citep{neuts1981matrixgeometric,latoucheramaswami1999introduction}.

To see the consistency of Equation~\eqref{eq:explicit-R}, let
\[
\mathbf{M}
=
\lambda\mathbf{I}_{n+1}
-
\lambda\mathbf{1}\boldsymbol{\alpha}
-
\mathbf{T}.
\]
Since
$\boldsymbol{\alpha}\mathbf{1}=1$ and
$\mathbf{t}=-\mathbf{T}\mathbf{1}$,
we have $\mathbf{M}\mathbf{1}=\mathbf{t}$ and therefore
$\mathbf{R}\mathbf{t}=\lambda\mathbf{1}$.
Substitution then verifies the matrix-quadratic equation above.

Under $\rho<1$, the stationary level vectors are
\begin{equation}
\boldsymbol{\pi}_k
=
(1-\rho)\boldsymbol{\alpha}\mathbf{R}^{k},
\qquad
k\ge1,
\label{eq:matrix_geometric}
\end{equation}
and hence
\[
\Pr\{N=k\}
=
\boldsymbol{\pi}_k\mathbf{1}.
\]

\begin{corollary}[Congestion tail probability]
\label{cor:tail-prob}
For every integer $K\ge1$,
\begin{equation}
\Pr\{N\ge K\}
=
(1-\rho)
\boldsymbol{\alpha}
\mathbf{R}^{K}
(\mathbf{I}-\mathbf{R})^{-1}
\mathbf{1}.
\label{eq:tail-prob}
\end{equation}
\end{corollary}

The moment formulas characterize mean congestion through $m_1$ and
$m_2$, whereas Equation~\eqref{eq:tail-prob} retains the stationary
queue-length tail.
\subsection{Numerical validation}
\label{subsec:general_n_validation}
The matrix-analytic formulation is verified along three complementary
directions: cross-validation against two independent analytical
derivations, sample-path validation via discrete-event simulation,
and a structural experiment that disentangles the effect of treatment
heterogeneity from the effect of treatment multiplicity. The
distributional check and the heterogeneity experiment are summarized
graphically in Figure~\ref{fig:general_n_pmf} and (b),
respectively.
Unless stated otherwise, $\lambda=2.5$, $\beta=0.4$, $\mu_D=8$, and
the direct-referral and post-diagnostic routing probabilities are
balanced,
\begin{equation}
\beta_i=\frac{1-\beta}{n},
\qquad
\gamma_i=\frac{1}{n},
\qquad i=1,\ldots,n.
\label{eq:balanced_general_n_setting}
\end{equation}
Treatment heterogeneity is controlled by
\begin{equation}
\mu_{T_i}=\bar{\mu}_T(1+h x_i),
\qquad
x_i=-1+\frac{2(i-1)}{n-1},
\label{eq:general_n_treatment_rates}
\end{equation}
with baseline $\bar{\mu}_T=6$ and $h=0.4$. All configurations satisfy
$\rho<1$.
Because Algorithm~\ref{alg:general_n} reduces to the closed-form PGFs
at $n=2$ and to the P-K formula at any $n$, exact
agreement across the three routes is a stringent correctness test.
Table~\ref{tab:general_n_matrix_validation} reports the results.

\begin{figure}[htbp]
\centering
\includegraphics[width=0.68\textwidth]
{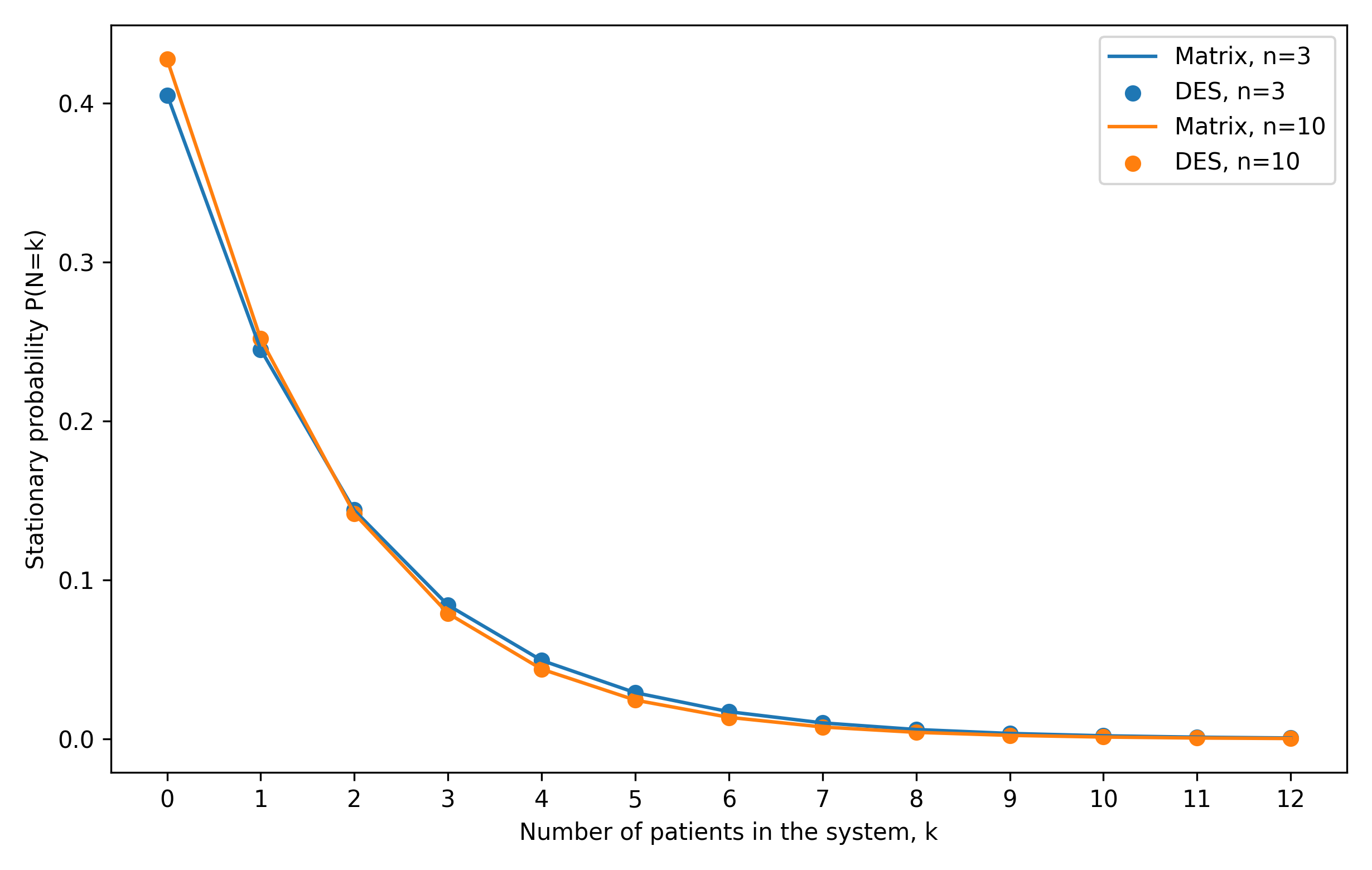}
\caption{Stationary queue-size distribution from the
matrix-geometric solution (lines) and discrete-event simulation
(bars) for representative arbitrary-$n$ configurations.}
\label{fig:general_n_pmf}
\end{figure}
For $n=2$, the matrix-geometric values reproduce the phase-resolved
PGF results. For $n>2$, the evaluated
matrix-geometric and Pollaczek--Khintchine means agree to numerical
precision used as an consistency check,
while the discrete-event simulation provides an
sample-path benchmark.
Algorithm~\ref{alg:general_n} reproduces the exact
two-treatment closed form. At $n>2$, the matrix-analytic and moment
results agree to $10^{-13}$, providing an internal consistency check
in the absence of a closed-form PGF baseline.
\begin{table}[htbp]
\centering
\caption{Cross-validation of PGF, matrix-analytic (MA), and
 P-K derivations.
$\varepsilon_L=|L_{\mathrm{MA}}-L_{\mathrm{PK}}|/L_{\mathrm{PK}}$;
$\varepsilon_{\mathrm{flow}}=|TH-\lambda|/\lambda$.}
\label{tab:general_n_matrix_validation}
\small
\resizebox{\textwidth}{!}{%
\begin{tabular}{c c c c c c c c c}
\toprule
$n$ & $\rho$ & $L_{\mathrm{PGF}}$ & $L_{\mathrm{MA}}$
& $L_{\mathrm{PK}}$ & $W_{\mathrm{MA}}$
& $W_{\mathrm{PK}}$ & $\varepsilon_L$
& $\varepsilon_{\mathrm{flow}}$ \\
\midrule
2  & 0.621032 & 1.640857 & 1.640857 & 1.640857
   & 0.656343 & 0.656343
   & $3.39\!\times\!10^{-13}$ & $1.07\!\times\!10^{-13}$ \\
3  & 0.594577 & --       & 1.447776 & 1.447776
   & 0.579110 & 0.579110
   & $3.23\!\times\!10^{-13}$ & $1.08\!\times\!10^{-13}$ \\
5  & 0.580357 & --       & 1.350623 & 1.350623
   & 0.540249 & 0.540249
   & $2.51\!\times\!10^{-13}$ & $8.63\!\times\!10^{-14}$ \\
10 & 0.572472 & --       & 1.299072 & 1.299072
   & 0.519629 & 0.519629
   & $2.49\!\times\!10^{-13}$ & $8.74\!\times\!10^{-14}$ \\
\bottomrule
\end{tabular}}
\end{table}

\subsection{Verification through DES}
\label{subsec:des_verification}
An event-driven simulator of the underlying CTMC is used for
$n\in\{3,5,10\}$, with 20 replications of $50{,}000$ time units each
after a $2{,}000$-unit warm-up. The simulation follows an independent event-scheduling
implementation and is used as an external numerical benchmark
following standard discrete-event simulation practice
\citep{law2015simulation}.
\begin{table}[htbp]
\centering
\caption{Analytical versus DES results for systems with more than two
treatment modes.}
\label{tab:general_n_des_validation}
\small
\begin{tabular}{c c c c c c}
\toprule
$n$ & Metric & Analytical & DES mean & 95\% CI & Error (\%) \\
\midrule
3  & $L$  & 1.447776 & 1.449404 & $[1.443011,\,1.455797]$ & 0.112 \\
3  & $W$  & 0.579110 & 0.579823 & $[0.577666,\,0.581980]$ & 0.123 \\
3  & $TH$ & 2.500000 & 2.499668 & $[2.496323,\,2.503013]$ & 0.013 \\
\midrule
5  & $L$  & 1.350623 & 1.353566 & $[1.347005,\,1.360127]$ & 0.218 \\
5  & $W$  & 0.540249 & 0.541152 & $[0.539024,\,0.543281]$ & 0.167 \\
5  & $TH$ & 2.500000 & 2.501230 & $[2.497766,\,2.504694]$ & 0.049 \\
\midrule
10 & $L$  & 1.299072 & 1.298844 & $[1.291355,\,1.306332]$ & 0.018 \\
10 & $W$  & 0.519629 & 0.519764 & $[0.517136,\,0.522392]$ & 0.026 \\
10 & $TH$ & 2.500000 & 2.498863 & $[2.496053,\,2.501673]$ & 0.045 \\
\bottomrule
\end{tabular}
\end{table}
Every analytical value lies within the corresponding 95\% confidence
interval, with maximum relative error $0.22\%$ (Table~\ref{tab:general_n_des_validation}).
Figure~\ref{fig:general_n_pmf} complements the first-moment
agreement with a distributional comparison: the empirical queue-size
histogram matches the matrix-geometric probabilities across the full
support, ruling out compensating errors in higher moments.
\section{Congestion Analysis}
\label{sec:structural}

The preceding analysis provides both mean performance measures and the full
stationary queue-length distribution. We now use these quantities to identify
which features of the service structure create congestion and to prepare the
capacity decisions of Section~\ref{sec:decision}.

\subsection{Treatment heterogeneity and  multiplicity}

Figure~\ref{fig:general_n_heterogeneity} varies the treatment rates according
to Equation~\eqref{eq:general_n_treatment_rates} while preserving the same
average treatment-rate benchmark. At $h=0$, all treatment rates coincide and
the curves for $n\in\{3,5,10\}$ meet at the same value. Increasing $n$ alone
therefore does not create a congestion penalty under the balanced patient mix.
As $h$ increases, the treatment rates spread around their common benchmark.
Because both $1/\mu$ and $1/\mu^2$ are convex on the positive real line, this
spread increases the first and second service-time moments under the balanced
construction and consequently increases $\rho$ and $L$. The numerical effect
is stronger for small $n$, where the grid places more weight on the extreme
rates. Thus, in this experiment it is treatment-rate heterogeneity, rather
than the number of treatment labels by itself, that drives the deterioration
in performance.

\subsection{Mode-specific workload and marginal congestion effect}
\label{sec:workload}
Let $\mathcal J=\{D,T_1,\ldots,T_n\}$ denote the service modes and use the
weights $\omega_D=\beta$ and $\omega_{T_i}=\omega_i$. For a baseline rate
vector $\boldsymbol\mu^{(0)}$, the effective load decomposes as
\begin{equation}
\rho^{(0)}
=\lambda\sum_{j\in\mathcal J}\frac{\omega_j}{\mu_j^{(0)}}
=\sum_{j\in\mathcal J}\theta_j,
\qquad
\theta_j=\frac{\lambda\omega_j}{\mu_j^{(0)}}.
\label{eq:workload-decomp}
\end{equation}
The quantities $\theta_j$ identify where the provider's utilization is consumed.
They are not, however, sufficient to rank the marginal effect of a rate
increase on queue length because $L$ depends on the second service-time moment
as well as on $\rho$. To make this distinction explicit, define the
scale-free congestion elasticity
\begin{equation}
\mathcal E_j
=-\frac{\mu_j}{L}\frac{\partial L}{\partial\mu_j}.
\label{eq:elasticity}
\end{equation}
In the reciprocal variables $y_D=1/\mu_D$ and $y_i=1/\mu_{T_i}$, let
\begin{equation}
q(\mathbf y)=\beta y_D^2+\beta\sum_{i=1}^n\gamma_i y_Dy_i+
\sum_{i=1}^n\omega_i y_i^2,
\label{eq:q-structural}
\end{equation}
so that $m_2=2q(\mathbf y)$ and
$L=\rho+\lambda^2q/(1-\rho)$. Hence
\begin{align}
\frac{\partial L}{\partial y_D}
&=\lambda\beta+\lambda^2\left[
\frac{2\beta y_D+\beta\sum_i\gamma_i y_i}{1-\rho}
+\frac{\lambda\beta q}{(1-\rho)^2}\right],
\label{eq:sens-D}\\
\frac{\partial L}{\partial y_i}
&=\lambda\omega_i+\lambda^2\left[
\frac{\beta\gamma_i y_D+2\omega_i y_i}{1-\rho}
+\frac{\lambda\omega_i q}{(1-\rho)^2}\right].
\label{eq:sens-T}
\end{align}
Since $-\partial L/\partial\mu_j=(1/\mu_j^2)\partial L/\partial y_j$,
Equations~\eqref{eq:sens-D}--\eqref{eq:sens-T} show explicitly why the
largest workload contribution need not be the only relevant criterion for a
capacity decision.

For the baseline used in the numerical study,
$(\lambda_D,\lambda_{T_1},\lambda_{T_2})=(1.5,0.5,0.7)$,
$(\mu_D,\mu_{T_1},\mu_{T_2})=(8,5,7)$, and
$(\gamma_1,\gamma_2)=(0.6,0.4)$, the values are shown in
Table~\ref{tab:workload-elasticity}. Treatment mode $T_1$ has both the largest
workload contribution and the largest local congestion elasticity, while the
magnitudes for $D$ and $T_2$ differ despite their similar workload shares.
\begin{table}[htbp]
\centering
\caption{Mode-specific workload and congestion elasticity at the baseline configuration.}
\label{tab:workload-elasticity}
\begin{tabular}{lccc}
\toprule
Mode & $\theta_j$ & $-\partial L/\partial\mu_j$ & $\mathcal E_j$ \\
\midrule
$D$   & 0.187500 & 0.174620 & 0.8043 \\
$T_1$ & 0.280000 & 0.440428 & 1.2679 \\
$T_2$ & 0.185714 & 0.180369 & 0.7269 \\
\bottomrule
\end{tabular}
\end{table}

\section{Capacity Intervention and Design}
\label{sec:decision}

The structural measures above separate two related decisions. The first is a
targeted intervention in which only one mode-specific service rate can be
improved. The second is a strategic redesign in which all service rates may be
chosen jointly. We first derive the targeted rule from the workload
decomposition, then refine it with the stationary congestion tail, and finally
use the convex cost model as the joint economic benchmark.

\subsection{Utilization-targeted single-mode intervention}
\label{sec:load-triggered}
Suppose the effective service rate of mode $j\in\mathcal J$ is increased by
$\delta_j\ge0$ while all other rates remain at their baseline values. Define
\begin{equation}
A_{-j}=\sum_{\ell\in\mathcal J,\,\ell\ne j}
\frac{\omega_\ell}{\mu_\ell^{(0)}}.
\label{eq:residual}
\end{equation}
The resulting utilization is
\begin{equation}
\rho_j(\delta_j)=\lambda\left(A_{-j}+\frac{\omega_j}{\mu_j^{(0)}+\delta_j}\right).
\label{eq:rho-controlled}
\end{equation}
Let $\bar\rho\in(0,1)$ be a prescribed utilization target.

\begin{proposition}
\label{prop:capacity}
A finite single-mode adjustment of mode $j$ can attain
$\rho_j(\delta_j)\le\bar\rho$ if and only if
$\lambda A_{-j}<\bar\rho$. Whenever this condition holds, the minimum
nonnegative increment is
\begin{equation}
\delta_{j,\bar\rho}
=\left[\frac{\omega_j}{\bar\rho/\lambda-A_{-j}}-\mu_j^{(0)}\right]^+.
\label{eq:delta-rho}
\end{equation}
\end{proposition}
\begin{proof}
The mapping in Equation~\eqref{eq:rho-controlled} is continuous and strictly
decreasing in $\delta_j$, with limiting utilization $\lambda A_{-j}$ as
$\delta_j\to\infty$. The feasibility condition and Equation~\eqref{eq:delta-rho}
then follow by solving $\rho_j(\delta_j)=\bar\rho$ and imposing
$\delta_j\ge0$.
\end{proof}

Let $\Delta\rho=\rho^{(0)}-\bar\rho>0$. Since
$\rho^{(0)}=\theta_j+\lambda A_{-j}$, every feasible phase with
$\theta_j>\Delta\rho$ satisfies
\begin{equation}
\frac{\delta_{j,\bar\rho}}{\mu_j^{(0)}}
=\frac{\Delta\rho}{\theta_j-\Delta\rho},
\label{eq:delta-share}
\end{equation}
so the proportional improvement required to reach the same utilization target
is strictly decreasing in $\theta_j$. This relation identifies the amount of rate
improvement, not its economic desirability. If $C_{\mu_j}$ denotes the
marginal cost of improving rate $j$, the minimum direct investment is
\begin{equation}
j_\rho^*\in\arg\min_{j\in\mathcal F_\rho}
C_{\mu_j}\delta_{j,\bar\rho},
\qquad
\mathcal F_\rho=\{j:\lambda A_{-j}<\bar\rho\}.
\label{eq:min-investment}
\end{equation}
For two feasible modes $j$ and $\ell$, mode $j$ is cheaper precisely when
\begin{equation}
\frac{C_{\mu_j}}{C_{\mu_\ell}}
<\frac{\delta_{\ell,\bar\rho}}{\delta_{j,\bar\rho}},
\label{eq:break-even-cost}
\end{equation}
which gives a break-even cost ratio without requiring an arbitrary cost
scenario.

\subsection{Utilization-equivalent interventions and delay}
\label{sec:equal-load}
The rule above controls the first service-time moment because
$\rho=\lambda m_1$. It does not fix the second moment.

\begin{proposition}
\label{prop:equal-load}
Consider two feasible single-mode interventions $j$ and $\ell$ that both
attain the same target utilization $\bar\rho$. Then
$m_{1,j}=m_{1,\ell}=\bar\rho/\lambda$, and
\begin{equation}
m_{2,j}<m_{2,\ell}
\quad\Longleftrightarrow\quad
L_j<L_\ell
\quad\Longleftrightarrow\quad
W_j<W_\ell.
\label{eq:equal-load-order}
\end{equation}
\end{proposition}
\begin{proof}
With $\rho_j=\rho_\ell=\bar\rho$, Equation~\eqref{eq:L_compact} gives
$L_r=\bar\rho+\lambda^2m_{2,r}/[2(1-\bar\rho)]$ for
$r\in\{j,\ell\}$. The coefficient multiplying $m_{2,r}$ is positive, so the
ordering of $L$ is exactly the ordering of $m_2$. Since $W=L/\lambda$, the
same ordering holds for $W$.
\end{proof}

Proposition~\ref{prop:equal-load} provides the link between the analytical
capacity rule and the distributional analysis: equal utilization does not
imply equal delay, and it need not imply equal protection against a large
queue. The latter requires the full stationary distribution.

\subsection{Distribution-aware intervention}
\label{sec:risk-rule}
For a congestion threshold $K$ and a prescribed tail probability
$\varepsilon$, define
\begin{equation}
g_j(\delta)=\Pr\{N\ge K;\,\mu_j^{(0)}+\delta\},
\qquad
0\le\delta\le\Delta_j^{\max},
\label{eq:risk-function}
\end{equation}
where $\Delta_j^{\max}=\bar\mu_j-\mu_j^{(0)}$ is the largest implementable
increase in mode $j$. The minimum risk-targeted increment is
\begin{equation}
\delta^{\mathrm{risk}}_{j,K,\varepsilon}
=\inf\{\delta\in[0,\Delta_j^{\max}]:g_j(\delta)\le\varepsilon\}.
\label{eq:delta-risk}
\end{equation}
The tail probability is evaluated from Corollary~\ref{cor:tail-prob} after
recomputing the rate matrix at the candidate service rate.

\begin{remark}
Fix the arrival sequence, patient types, routing outcomes, and all service
rates except $\mu_j$. Couple two systems by using the same uniform random
variable for each service phase. A larger $\mu_j$ weakly shortens the
phase-$j$ duration for every patient who visits that phase and leaves all
other phase durations unchanged. Hence the total service requirement of each
patient is no larger in the faster system. Under FCFS, the departure epochs
satisfy the usual recursion $D_n=\max\{D_{n-1},A_n\}+S_n$
\citep{baccelli2013_elements_queueing}. An induction therefore gives
$D_n^{\mathrm{fast}}\le D_n^{\mathrm{base}}$ for every customer, so with the
same arrival counting process the number in system is pathwise no larger in
the faster system. Consequently $g_j(\delta)$ is non-increasing. A mode is
risk-feasible within the available capacity if
$g_j(\Delta_j^{\max})\le\varepsilon$; for a feasible mode the minimum
increment in Equation~\eqref{eq:delta-risk} is obtained by a one-dimensional
bisection search.
\end{remark}

The final targeted intervention enforces both the utilization and congestion
requirements:
\begin{equation}
\delta_j^{H}
=\max\{\delta_{j,\bar\rho},\delta^{\mathrm{risk}}_{j,K,\varepsilon}\},
\qquad
j_H^*\in\arg\min_{j\in\mathcal F_H}C_{\mu_j}\delta_j^H,
\label{eq:hybrid-rule}
\end{equation}
where the jointly feasible mode set is
\begin{equation}
\mathcal F_H=
\left\{j\in\mathcal J:
\lambda A_{-j}<\bar\rho,\;
\delta_{j,\bar\rho}\le\Delta_j^{\max},\;
 g_j(\Delta_j^{\max})\le\varepsilon
\right\}.
\label{eq:FH-set}
\end{equation}
Thus a mode must be capable of reaching both the utilization target and the
tail-risk target within its available rate improvement. Equation~\eqref{eq:hybrid-rule}
uses the closed-form load adjustment as the analytical component and invokes
the exact matrix-geometric tail only for the distributional correction.

\subsection{The cost function}
\label{sec:joint-cost}
For a strategic redesign in which all mode-specific rates may change jointly,
consider
\begin{align}
TC(\boldsymbol\mu)
={}&C_hL(\boldsymbol\mu)
+C_D\frac{\lambda\beta}{\mu_D}
+\sum_{i=1}^n C_{T_i}\frac{\lambda\omega_i}{\mu_{T_i}}
+C_{\mu_D}\mu_D+\sum_{i=1}^nC_{\mu_{T_i}}\mu_{T_i},
\label{eq:total-cost-new}
\end{align}
where $C_h>0$ is the holding cost per patient per unit time,
$C_D,C_{T_i}\ge0$ are phase-active operating costs, and
$C_{\mu_D},C_{\mu_{T_i}}>0$ are the amortized costs of maintaining the
corresponding effective service rates. These coefficients are illustrative
planning weights rather than estimates from field data.

Let $y_D=1/\mu_D$ and $y_i=1/\mu_{T_i}$. Then
\begin{equation}
\rho(\mathbf y)=\lambda\left(\beta y_D+\sum_{i=1}^n\omega_i y_i\right),
\qquad
q(\mathbf y)=\beta y_D^2+\beta\sum_{i=1}^n\gamma_i y_Dy_i+
\sum_{i=1}^n\omega_i y_i^2,
\label{eq:rho-q-cost}
\end{equation}
and
\begin{equation}
L(\mathbf y)=\rho(\mathbf y)+\frac{\lambda^2q(\mathbf y)}{1-\rho(\mathbf y)}.
\label{eq:L-cost-new}
\end{equation}
The joint design is
\begin{equation}
\begin{aligned}
TC^{\mathrm{joint}}=\min_{\mathbf y}\quad&TC(\mathbf y)\\
\text{s.t.}\quad&\rho(\mathbf y)\le\bar\rho,\\
&1/\bar\mu_D\le y_D\le1/\underline\mu_D,\\
&1/\bar\mu_{T_i}\le y_i\le1/\underline\mu_{T_i},\quad i=1,\ldots,n.
\end{aligned}
\label{eq:joint-cost-new}
\end{equation}
The feasible set is nonempty if
\begin{equation}
\lambda\left(\frac{\beta}{\bar\mu_D}+
\sum_{i=1}^n\frac{\omega_i}{\bar\mu_{T_i}}\right)\le\bar\rho.
\label{eq:joint-feasible-new}
\end{equation}

\begin{proposition}
\label{prop:joint-convex-new}
Under Equation~\eqref{eq:joint-feasible-new} and positive capacity-cost
coefficients, problem~\eqref{eq:joint-cost-new} is strictly convex in
$\mathbf y$ and therefore has a unique global minimizer.
\end{proposition}
\begin{proof}
Using $\omega_i=\beta_i+\beta\gamma_i$ and $\sum_i\gamma_i=1$,
\[
q(\mathbf y)=\sum_{i=1}^n\beta_i y_i^2+
\beta\sum_{i=1}^n\gamma_i(y_D^2+y_Dy_i+y_i^2).
\]
Each term $y_D^2+y_Dy_i+y_i^2=(y_D+y_i/2)^2+3y_i^2/4$ is convex, so
$q$ is a positive-semidefinite quadratic form. The perspective
$q(\mathbf y)/t$ is convex for $t>0$; taking
$t=1-\rho(\mathbf y)$ gives convexity of the congestion term on the stable
feasible set \citep{boydvandenberghe2004convex}. The phase-active terms are
affine in $\mathbf y$, while $C_{\mu_j}/y_j$ is strictly convex for positive
$y_j$. Their sum is therefore strictly convex. Compactness of the rate box
ensures attainment of the unique minimum.
\end{proof}

For an interior optimum, with $\eta^*\ge0$ denoting the multiplier of the
utilization constraint, the KKT conditions take the form
\begin{align}
\frac{C_{\mu_D}}{(y_D^*)^2}
&=C_h\frac{\partial L}{\partial y_D}(\mathbf y^*)
+\lambda C_D\beta+\eta^*\lambda\beta,\\
\frac{C_{\mu_{T_i}}}{(y_i^*)^2}
&=C_h\frac{\partial L}{\partial y_i}(\mathbf y^*)
+\lambda C_{T_i}\omega_i+\eta^*\lambda\omega_i,
\qquad i=1,\ldots,n,
\label{eq:KKT-new}
\end{align}
with complementary slackness
$\eta^*[\rho(\mathbf y^*)-\bar\rho]=0$. Thus the optimum balances the
marginal cost of capacity against the marginal reduction in congestion and
phase-active operating time.

\subsection{Restricted one-mode design}
\label{sec:restricted}
When only one mode can be improved during the planning period, let the other
rates remain at the baseline $\boldsymbol\mu^{(0)}$ and restrict
$\mu_j\in[\mu_j^{(0)},\bar\mu_j]$. For every feasible mode,
\begin{equation}
TC_j^{\mathrm{one}}=
\min_{\mu_j\in[\mu_j^{(0)},\bar\mu_j]}
TC(\mu_j,\boldsymbol\mu_{-j}^{(0)})
\quad\text{s.t.}\quad \rho\le\bar\rho.
\label{eq:one-mode-cost-new}
\end{equation}
This is a one-dimensional strictly convex problem in $y_j=1/\mu_j$. Define
\begin{equation}
TC^{\mathrm{one}}=\min_jTC_j^{\mathrm{one}},
\qquad
G_{\mathrm{restrict}}=100\frac{TC^{\mathrm{one}}-TC^{\mathrm{joint}}}
{TC^{\mathrm{joint}}}.
\label{eq:restriction-gap-new}
\end{equation}
Because the joint feasible set contains every feasible one-mode design,
$G_{\mathrm{restrict}}\ge0$. The quantity measures the economic value of being
able to redesign all service-mode efficiencies instead of changing only one.

\subsection{Numerical decision analysis}
\label{sec:decision-numerical}
The numerical experiments are organized to illustrate the decision rules and
to test structural implications rather than to assign universal importance to
one chosen parameter set. The baseline configuration is
$(\lambda_D,\lambda_{T_1},\lambda_{T_2})=(1.5,0.5,0.7)$,
$(\mu_D,\mu_{T_1},\mu_{T_2})=(8,5,7)$, and
$(\gamma_1,\gamma_2)=(0.6,0.4)$, which gives $\rho^{(0)}=0.653214$.

\subsubsection{Illustration of the closed-form intervention}
To demonstrate Proposition~\ref{prop:capacity} without attaching general
meaning to a single threshold, we reduce the baseline utilization by 15\%, so
$\bar\rho=0.85\rho^{(0)}=0.555232$. Table~\ref{tab:load-intervention}
reports the required rate increment for each feasible mode. With the
illustrative marginal rate costs $(8,10,12)$, treatment mode $T_1$ requires the
smallest direct investment.
\begin{table}[htbp]
\centering
\caption{Illustrative single-mode intervention for a 15\% reduction in baseline utilization.}
\label{tab:load-intervention}
\begin{tabular}{lrrrr}
\toprule
Mode & $\theta_j$ & $\delta_{j,\bar\rho}$ & $\delta_{j,\bar\rho}/\mu_j^{(0)}$ & $C_{\mu_j}\delta_{j,\bar\rho}$\\
\midrule
$D$   & 0.187500 & 8.7564 & 1.0946 & 70.05\\
$T_1$ & 0.280000 & 2.6916 & 0.5383 & 26.92\\
$T_2$ & 0.185714 & 7.8178 & 1.1168 & 93.81\\
\bottomrule
\end{tabular}
\end{table}
For the selected $T_1$ intervention, the rate changes from $5$ to $7.6916$,
while $W$ decreases from $0.6433$ to $0.4282$. These percentages are an
illustration of the analytical rule for this instance; the general result is
Equation~\eqref{eq:delta-share}, not the numerical reduction itself.

\subsubsection{Equal utilization does not imply equal congestion}
We next use the high-load scenario S5 from Table~\ref{tab:validation}, for
which $\rho^{(0)}=0.922429$, and independently adjust $D$, $T_1$, or $T_2$ to
three common target utilizations. Table~\ref{tab:equal-load-results} shows that
the ordering predicted by Proposition~\ref{prop:equal-load} persists across
all three targets. At $\bar\rho=0.80$, for example, the tail probability
$\Pr\{N\ge10\}$ ranges from $0.0786$ under the $T_1$ intervention to $0.0992$
under the $T_2$ intervention even though the final utilization is identical.
The relative difference between these two tail probabilities is about 26\%.
\begin{table}[htbp]
\centering
\caption{Utilization-equivalent interventions in the high-load S5 scenario.}
\label{tab:equal-load-results}
\small
\begin{tabular}{ccrrrr}
\toprule
$\bar\rho$ & Mode & $\delta_j$ & $m_2$ & $W$ & $\Pr\{N\ge10\}$\\
\midrule
0.85 & $D$   & 2.8604 & 0.09072 & 1.3583 & 0.1631\\
     & $T_1$ & 1.1192 & 0.08869 & 1.3329 & 0.1572\\
     & $T_2$ & 2.8324 & 0.09410 & 1.4004 & 0.1723\\
\midrule
0.80 & $D$   & 6.4195 & 0.08167 & 0.9775 & 0.0863\\
     & $T_1$ & 2.2376 & 0.07840 & 0.9469 & 0.0786\\
     & $T_2$ & 6.6434 & 0.08775 & 1.0344 & 0.0992\\
\midrule
0.75 & $D$   & 13.4485 & 0.07383 & 0.7528 & 0.0461\\
     & $T_1$ & 3.8562 & 0.06946 & 0.7201 & 0.0386\\
     & $T_2$ & 15.2785 & 0.08292 & 0.8208 & 0.0602\\
\bottomrule
\end{tabular}
\end{table}
Thus, utilization is a useful load target but not a complete description of
the post-intervention congestion distribution. The service mode used to obtain
the target also changes the second moment and the queue tail.

\subsubsection{The distributional criterion }
For the same S5 scenario, Table~\ref{tab:risk-results} fixes $K=10$ and varies
the prescribed tail probability over three levels. The minimum rate increments
are obtained by bisection on Equation~\eqref{eq:delta-risk}. A fixed
congestion-risk target does not correspond to a universal utilization level:
for $\varepsilon=0.08$, the required final utilizations are approximately
$0.794$, $0.801$, and $0.779$ for the $D$, $T_1$, and $T_2$ interventions,
respectively. Hence the adequacy of a utilization-only rule depends on which
service mode is improved.
\begin{table}[htbp]
\centering
\caption{Minimum service-rate increments for $\Pr\{N\ge10\}\le\varepsilon$ in the S5 scenario.}
\label{tab:risk-results}
\small
\begin{tabular}{ccrr}
\toprule
$\varepsilon$ & Mode & $\delta^{\mathrm{risk}}_{j,10,\varepsilon}$ & resulting $\rho$\\
\midrule
0.10 & $D$   & 5.3995 & 0.8116\\
     & $T_1$ & 1.8108 & 0.8171\\
     & $T_2$ & 6.5622 & 0.8008\\
\midrule
0.08 & $D$   & 7.0042 & 0.7941\\
     & $T_1$ & 2.2050 & 0.8012\\
     & $T_2$ & 9.2600 & 0.7792\\
\midrule
0.05 & $D$   & 12.1592 & 0.7566\\
     & $T_1$ & 3.1929 & 0.7681\\
     & $T_2$ & 22.9979 & 0.7297\\
\bottomrule
\end{tabular}
\end{table}

\subsubsection{Joint versus restricted cost}
For the cost study we use the baseline arrival and routing parameters above,
rate box $[8,16]\times[4,10]\times[5,12]$, and a nonbinding safety cap
$\bar\rho=0.80$. The two cost vectors are
$(110,65,75,85,8,10,12)$ and $(100,60,80,85,10,13,15)$ for
$(C_h,C_D,C_{T_1},C_{T_2},C_{\mu_D},C_{\mu_{T_1}},C_{\mu_{T_2}})$.
The unique joint optima and the best one-mode designs are reported in
Table~\ref{tab:joint-restricted}. In both cost settings, changing only $T_1$
is the best restricted design, but the restriction still creates a positive
cost penalty relative to the joint redesign.
\begin{table}[htbp]
\centering
\caption{Joint and restricted economic designs.}
\label{tab:joint-restricted}
\small
\resizebox{\textwidth}{!}{%
\begin{tabular}{clrrrrr}
\toprule
Cost set & Design & $\mu_D$ & $\mu_{T_1}$ & $\mu_{T_2}$ & $TC$ & $G_{\mathrm{restrict}}$(\%)\\
\midrule
I & Joint & 9.4943 & 8.1547 & 7.2930 & 388.9343 & --\\
  & Best one-mode ($T_1$) & 8.0000 & 8.5095 & 7.0000 & 392.4621 & 0.907\\
\midrule
II & Joint & 8.7637 & 7.4555 & 6.7492 & 441.1241 & --\\
   & Best one-mode ($T_1$) & 8.0000 & 7.5535 & 7.0000 & 442.1475 & 0.232\\
\bottomrule
\end{tabular}}
\end{table}
The restricted-design gap is small for these illustrative cost settings, but it
is not zero. More importantly, Equation~\eqref{eq:restriction-gap-new}
provides a direct measure of the economic value of allowing coordinated
changes across service modes; its magnitude can be evaluated under any
institution-specific cost coefficients and capacity bounds.

\section{Conclusion and Future Research}
\label{sec:conclusion}

This paper develops capacity framework for a pooled
single-server diagnostic--treatment service with first-time and referred
patients. The patient pathway leads to a phase-type service
requirement: first-time patients consume diagnosis followed immediately by a
treatment phase, whereas referred patients consume treatment only. Closed-form
PGFs for the two-treatment case, and a matrix-geometric solution
extends the stationary queue-length distribution to arbitrary $n$. The
analytical and matrix-analytic calculations agree with the
Pollaczek--Khintchine mean-value relations and with DES method.
The utilization decomposition identifies
how much load is contributed by each service mode and yields a closed-form
minimum improvement for a prescribed utilization target. Equal utilization,
however, does not imply equal congestion performance: mode-specific
interventions can produce different second service-time moments and therefore
different mean delays. The distribution-aware rule uses the
matrix-geometric tail to identify the mean-load intervention when a severe
congestion target is binding. The cost function is strictly convex while the restricted one-mode problem quantifies the economic value
of coordinated redesign.The restricted-design gap provides a further measure of how
much is lost when only one service mode can be improved during a planning
period.

The model deliberately uses a single pooled provider, stationary Poisson
arrivals, and exponential phase durations to retain analytical and
distributional tractability. The exponential assumption imposes unit squared
coefficient of variation within each service phase and may not match all
clinical service-time data. The model is also a planning abstraction rather
than a calibrated representation of a specific hospital. Natural extensions
include general phase-type or data-fitted service distributions, multiple
parallel providers, abandonment and priority rules, nonstationary arrivals,
and dynamic capacity policies. These extensions can build on the present
decision structure while relaxing the assumptions needed for the closed-form
and convex results.

\section*{Data availability} Not applicable, the study does not report any data. 
\section*{Declarations}
\subsection*{Conflict of interest} All the authors declare that they have no known conflicts of interest.




\bibliographystyle{elsarticle-num}
\bibliography{sn-bibliography}
\end{document}